\documentclass[12pt]{article}
\usepackage{amsmath,amssymb}
\usepackage[T1]{fontenc}
\usepackage{latexsym}
\usepackage{epsfig}
\usepackage{psfrag}
\usepackage[utf8]{inputenc}
\usepackage[a4paper, top=3cm, bottom=3cm ]{geometry}

\newtheorem{thm}{Theorem}[section]
\newtheorem{defn}{Definition}[section]

\newtheorem{prop}{Proposition}[section]

\title
{Lu Qi-Keng's problem for intersection of two complex ellipsoids}

\author{\normalsize Tomasz Beberok \\
\small Faculty of Mathematics and Computer Science, Jagiellonian University,\\
\small Lojasiewicza 6, 30-048 Krakow, Poland \\}

\date{}

\begin{document}

\begin{center}
  \textbf{Lu Qi-Keng's problem for intersection of two complex ellipsoids}
\end{center}
\vskip1em
\begin{center}
  Tomasz Beberok
\end{center}

\vskip2em

In this paper We investigate the Lu Qi-Keng problem for intersection of two complex ellipsoids $\{z \in \mathbb{C}^3 \colon |z_1|^2 + |z_2|^q < 1, \quad |z_1|^2 + |z_3|^r < 1\}$.
\vskip1em


\section{Introduction}

In 1921, S. Bergman introduced a kernel function, which is now known as the Bergman kernel function. It is well known that there exists a unique Bergman kernel function for each bounded domain in $\mathbb{C}^n$. Computation of the Bergman kernel function by explicit formulas is an important research direction in several
complex variables. Let $D$ be a bounded domain in $\mathbb{C}^n$. The Bergman space $L^2_a(D)$ is the space of all square integrable holomorphic functions on $D$. Then the Bergman kernel $K_D(z,w)$ is defined \cite{BE} by
\begin{align*}
K_D(z,w)= \sum_{j=0}^{\infty} \phi_j(z) \overline{\phi_j(w)}, \quad (z,w) \in D \times D,
\end{align*}
where $\{\phi_j(\cdot) \colon  j = 0, 1, 2, . . .\}$ is a complete orthonormal basis for $L^2_a(D)$. In \cite{Qi} Lu Qi-Keng indicated that in many concrete examples of bounded domains, $K_{D}(z,\ w)\neq 0$ for all $z,  w\in D$, and considered the open problem whether the above property is generally true. M. Skwarczynski \cite{SKWA} called this problem Lu Qi-Keng conjecture in 1969 and gave the following definition:

\begin{defn}
A domain $D \subset \mathbb{C}^n$ is called a Lu Qi-Keng domain if  $K_D(z,w)\neq 0$ for all $z,w \in D$.
\end{defn}

Obviously, a biholomorphic image of a Lu Qi-Keng domain is a Lu Qi-Keng domain due to the rule of the Bergman kernel
transformation between two biholomorphic equivalent domains. A Cartesian product of two Lu Qi-Keng domains is a Lu
Qi-Keng domain. If $K_D \neq const$ and $D$ is the sum of an increasing sequence of Lu Qi-Keng domains $D_m$, then $D$ is a Lu Qi-Keng domain due to the Ramadanov theorem and Hurwitz theorem.

However, it is not always easy to determine whether or not a given domain is Lu Qi-Keng domain. In 1969, M. Skwarczynski
\cite{SKWA} gave the first example that the Bergman kernel on an annulus in the complex plane $\Omega = \{r < |z| < 1\}$ has zeros
if $0 <r < e^{-2}$. Since then many counterexamples appeared. In 1996, Harold P. Boas \cite{Boas1} proved that the bounded domains of holomorphy in $\mathbb{C}^n$ whose Bergman kernel functions are zero-free form a nowhere dense subset (with respect to a variant of the Hausdorff distance) of all bounded domains of holomorphy. Thus, contrary to many expectations, it is the normal situation for the Bergman kernel function of a domain to have zeroes. For more details, see the survey articles \cite{BL} and \cite{Weiping}.

Among all the known counterexamples, the following complex ellipsoids
$$\Omega_{m,n}^{(p,q)}:=\left\{ (w,z) \in \mathbb{C}^m \times \mathbb{C}^n\colon \|w\|_m^{2p} + \|z\|_n^{2q} <1 \right\}$$
were under consideration frequently, where $\|\cdot\|_m$ is the standard Hermitian norm in complex Euclidean space and $p, q$ are positive real numbers. For instance, Boas, Fu and Straube \cite{BS} proved there exists a strongly convex domain in $\mathbb{C}^n (n > 2)$ which is not Lu Qi-Keng by computing the kernel for $\Omega^{(p,\frac{1}{2})}_{1,1}= \left\{(w, z) \in \mathbb{C}^2 \colon |w|^{2p} + |z| < 1\right\}$ has zeroes if and only if $1/p > 2$. Applying the deflation theorem stated in \cite{BS} we are led quickly to that $\Omega_{m,n} := \{(w, z) \in \mathbb{C}^m \times \mathbb{C}^n \colon  \sum_{k=1}^m |w|^2 + \sum_{k=1}^n |z| < 1\}$ is not Lu Qi-Keng iff $m + 2n > 4$ \cite{JP2}. Using this result, Nguy\^{e}n Vi\^{e}t Anh \cite{Anh} exhibited a strongly convex algebraic complete Reinhardt domain which is not Lu Qi-Keng in $\mathbb{C}^n$ for any $n \geq 3$. \newline
When $m = q = 1$ and $2p \geq 1$ is not an integer, as an application of a theorem due to M. Engli\v{s} \cite{En} and an improvement by B. Chen, Chen [10] proved there exists a constant $n(p)$ depending on p such that for all $n > n(p)$, the domain $\Omega^{p,1}_{1,n}=\{|w|^{2p} +|z_1|^2 + \cdots + |z_n|^2 < 1\}$ is not Lu Qi-Keng. A similar argument as in \cite{BS} immediately shows that $\{|w|^{2p} + |z_1| +\cdots +|z_n| < 1\}$ is not Lu Qi-Keng iff $n \geq [n(p)/2] + 1$, where $[n(p)/2]$ produces the integer part of $n(p)/2$. \newline
In both the above cases, how zeroes of the Bergman kernel depend on the increasing of the dimension of the vector
$z = (z_1, \ldots , z_n)$ are described. The purpose of this paper is to consider intersection of two complex ellipsoids $\{z=(z_1,z_2,z_3) \in \mathbb{C}^3 \colon |z_1|^2 + |z_2|^q < 1, \quad |z_1|^2 + |z_3|^r < 1\}$.

\section{Main results}
The following is the main theorem of this paper.
\begin{thm}[Main Theorem]\label{MT}
For any positive real numbers $q$ and $r$ the domain $$\{z \in \mathbb{C}^3 \colon |z_1|^2 + |z_2|^q < 1, \quad |z_1|^2 + |z_3|^r < 1\}$$ is a Lu Qi-Keng domain.
\end{thm}

\section{Bergman kernel}

For Reinhardt domains it is a standard method for computing the Bergman kernel  to use series representation, since we can choose $\phi_{\alpha} (z) = \frac{z^{\alpha}}{ \|z^{\alpha}\|}.$ Put $\Phi_{\alpha}(\zeta)= z_1^{\alpha_1} z_2^{\alpha_2} z_3^{\alpha_3} $. It is well known, that function $f$ holomorphic in a Reinhardt domain $D \subset \mathbb{C}^n$  has a “global” expansion into a Laurent series $f(z)=\sum_{\alpha \in \mathbb{Z}^n} a_{\alpha} z^{\alpha}$, $z \in D$ (see Proposition 1.7.15 (c) in \cite{JP}). Moreover if $D \cap (\mathbb{C}^{j-1} \times \{0\} \times \mathbb{C}^{n-j} ) \neq \emptyset $, $j=1,\ldots, n$ then $a_{\alpha}=0$ for $\alpha \in \mathbb{Z}^n \setminus \mathbb{Z}^n_{+}$ (see Proposition 1.6.5 (c) in \cite{JP}). Therefore  $\{\Phi_{\alpha} \}$ such that each $\alpha_i \geq 0$ is a complete orthogonal set for $L^2(D_{q,r}^p)$, where $$D_{q,r}^p:= \{z \in \mathbb{C}^3 \colon |z_1|^p + |z_2|^q < 1, \quad |z_1|^p + |z_3|^r < 1\}.$$


\begin{prop}\label{pr1} Let $\alpha_i \in \mathbb{Z}_{+}$ for $i=1,2,3$. Then, we have
$${\left\| z_1^{\alpha_1} z_2^{\alpha_2} z_3^{\alpha_3} \right\|}^2_{L^2(D_{q,r}^2)}= \frac{  \pi^3 \Gamma(\alpha_1 + 1) \Gamma(\frac{2\alpha_2 + 2}{q} +\frac{2\alpha_3 + 2}{r} + 1)  }{  (\alpha_2 + 1) (\alpha_3 +1 ) \Gamma( \frac{2\alpha_2 + 2}{q} +\frac{2\alpha_3 + 2}{r} + \alpha_1 + 2)} $$
\end{prop}

\begin{proof}
$$\|  z_1^{\alpha_1} z_2^{\alpha_2} z_3^{\alpha_3}  \|^2_{L^2(D_{q,r}^2)} = \int\limits_{D_{q,r}^2} |z_1|^{2\alpha_1} |z_2|^{2\alpha_2} |z_3|^{2\alpha_3}dV(z)$$
we introduce polar coordinate in each variable by putting $z_1=r_1 e^{i\theta_1}$, $z_2=r_2 e^{i\theta_2}$, $z_3=r_3 e^{i\theta_3}$. After doing so, and integrating out the angular variables we have

$$(2 \pi)^3 \int_0^1 \int_0^{(1-r_1^2)^{1/q}} \int_0^{(1-r_1^2)^{1/r}} r_1^{2\alpha_1 + 1} r_2^{2\alpha_2 + 1} r_3^{2\alpha_3 + 1}\, dr_1 dr_2 dr_3$$
Integrating out of $r_2$ and $r_3$ variables, we obtain

$$ \frac{(2 \pi)^3}{(2\alpha_2 + 2 ) (2\alpha_3 + 2) } \int_0^1  r_1^{2\alpha_1 + 1} (1-r_1^2)^{\frac{2\alpha_2 + 2}{q} +\frac{2\alpha_3 + 2}{r}} \, dr_1$$
After little calculation using well known fact $$\int_0^1 x^a(1-x^p)^b \,dx = \frac{\Gamma((a+1)/p) \Gamma(b+1)}{p  \Gamma((a+1)/p + b + 1) } ,$$  we obtain desired result.
\end{proof}

Now we discuss the Bergman kernel for $D_{q,r}^2$.

\begin{thm}\label{BK}
The Bergman kernel for $$\{z \in \mathbb{C}^3 \colon |z_1|^2 + |z_2|^q < 1, \quad |z_1|^2 + |z_3|^r < 1\}$$ is given by
\begin{align*}
K_{D_{q,r}^2} &((z_1, z_2, z_3),(w_1, w_2, w_3)) =  \\ &\frac{q r (1-\mu_2)(1-\mu_3)+ 2q (1-\mu_2) (1+\mu_3)+2 r (1+\mu_2) (1-\mu_3)}{ \pi^3 q r \left(1-\nu_1^2\right)^{2+2/q+2/r}  (1-\mu_2)^3 (1-\mu_3)^3 } ,
\end{align*}
where $\nu_1= z_1 \overline{w}_1$, $\nu_2= z_2 \overline{w}_2$, $\nu_3= z_3 \overline{w}_3$ and $\mu_2= \frac{\nu_2}{(1-\nu_1)^{2/q}}$, $\mu_3= \frac{\nu_3}{(1-\nu_1)^{2/r}}$.

\end{thm}
\begin{proof}

By series representation of the Bergman kernel function, we have

\begin{align*}
K_{D_{q,r}^2} &((z_1, z_2, z_3),(w_1, w_2, w_3)) = \frac{1}{\pi^3} \\&\sum_{\alpha_1, \alpha_2 , \alpha_3 =0 }^{ \infty} \frac{(\alpha_2 + 1) (\alpha_3 +1 ) \Gamma( \frac{2\alpha_2 + 2}{q} +\frac{2\alpha_3 + 2}{r} + \alpha_1 + 2) }{\Gamma(\alpha_1 + 1) \Gamma(\frac{2\alpha_2 + 2}{q} +\frac{2\alpha_3 + 2}{r} + 1)} \nu_1^{\alpha_1} \nu_2^{\alpha_2} \nu_3^{\alpha_3} ,
\end{align*}
where $\nu_1= z_1 \overline{w}_1$, $\nu_2= z_2 \overline{w}_2$, $\nu_3= z_3 \overline{w}_3$. \\
Sum out of $\nu_1$ variable, we have

\begin{align*}
\left(\frac{1}{1-\nu_1^2}\right)^{2+2/q+2/r} \sum_{ \alpha_2 , \alpha_3 =0 }^{ \infty} \frac{(\alpha_2 + 1) (\alpha_3 +1 ) \Gamma( \frac{2\alpha_2 + 2}{q} +\frac{2\alpha_3 + 2}{r}  + 2) }{\pi^3 \Gamma(\frac{2\alpha_2 + 2}{q} +\frac{2\alpha_3 + 2}{r} + 1)} \mu_2^{\alpha_2} \mu_3^{\alpha_3} ,
\end{align*}
where $\mu_2= \frac{\nu_2}{(1-\nu_1)^{2/q}}$, $\mu_3= \frac{\nu_3}{(1-\nu_1)^{2/r}}$.\\
Using the identity $\Gamma(a+1)=a \Gamma(a)$, after a little simplification, we obtain

\begin{align*}
 \sum_{ \alpha_2 , \alpha_3 =0 }^{ \infty} \frac{(\alpha_2 + 1) (\alpha_3 +1 ) \left(\frac{2\alpha_2 + 2}{q} +\frac{2\alpha_3 + 2}{r}  + 1 \right)}{\pi^3 \left(1-\nu_1^2\right)^{2+2/q+2/r} } \mu_2^{\alpha_2} \mu_3^{\alpha_3}
\end{align*}
After a little calculations, we have

\begin{align*}
 \frac{q r (1-\mu_2)(1-\mu_3)+ 2q (1-\mu_2) (1+\mu_3)+2 r (1+\mu_2) (1-\mu_3)}{\pi^3 q r \left(1-\nu_1^2\right)^{2+2/q+2/r}  (1-\mu_2)^3 (1-\mu_3)^3 }.
\end{align*}
\end{proof}

\section{Proof of the main theorem}
Note that the zero set is a bi-holomorphic invariant object. Since any point $(z_1,z_2,z_3)\in D_{q,r}^2$ can be mapped equivalently onto the form $(0,\widetilde{z_2},\widetilde{z_3})$ by following automorphism of the $D_{q,r}^2$  $$ D_{q,r}^2 \ni(z_1,z_2,z_3)\mapsto \left(\frac{z_1-a}{1-\overline{a}z_1}, \frac{(1-|a|^2)^{1/q}}{(1-\overline{a}z_1)^{2/q}}z_2, \frac{(1-|a|^2)^{1/r}}{(1-\overline{a}z_1)^{2/r}}z_3 \right) \in \mathbb{C}^3 .$$  Therefore, we need only consider the zeroes restricted to $\mathbb{D} \times \mathbb{D}$, where $\mathbb{D}:=\{z \in \mathbb{C}:  |z|<1 \}$ . Now by Theorem \ref{BK}
\begin{align*}
K_{D_{q,r}^2} &((0, z_2, z_3),(0, w_2, w_3)) =  \\ &\frac{q r (1-\nu_2)(1-\nu_3)+ 2q (1-\nu_2) (1+\nu_3)+2 r (1+\nu_2) (1-\nu_3)}{ \pi^3 q r   (1-\nu_2)^3 (1-\nu_3)^3 } ,
\end{align*}
where  $\nu_2= z_2 \overline{w}_2$, $\nu_3= z_3 \overline{w}_3$.

Denote by $$F(x,y)=q r (1-x)(1-y)+ 2q (1-x) (1+y)+2 r (1+x) (1-y),$$
then the Bergman kernel $K_{D_{q,r}^2}$ is zero free inside $D_{q,r}^2 \times D_{q,r}^2$ if and only if $F(x,y) \neq 0$ for all $(x,y) \in \mathbb{D} \times \mathbb{D}$. \newline
\indent Let us recall the stability criteria  for a real two-variable polynomial $$h(s,z)=\sum_{j=0}^n \sum_{k=0}^m h_{jk} s^j z^k$$
where $s,z \in \mathbb{C}$ are complex variables, and for some $j,k$ the coefficients $h_{jk}$ are not both zero. Polynomial $h(s, z)$ satisfies the stability property
\begin{equation}\label{sta}
h(s, z)\neq0, \quad (s,z) \in \overline{\mathbb{D}} \times \overline{\mathbb{D}},
\end{equation}
where $\overline{\mathbb{D}}$ is the closure of $\mathbb{D}$. \newline
By following Huang \cite{Hu}, one can show that \ref{sta} is equivalent to
\begin{eqnarray}
h(s,0) \neq 0, \quad \forall s \in \overline{\mathbb{D}} \label{sta1} \\ h(e^{iw},z) \neq 0, \quad \forall z \in \overline{\mathbb{D}}.\label{sta2}
\end{eqnarray}
Condition (\ref{sta1}) means that the new polynomial $f(s)=s^nh(s^{-1},0)$ has all zeros in the open unit circle $\mathbb{D}$, that is, $f(s)$ is $\mathbb{D}$-stable. To test condition (\ref{sta2}), we consider $d(z)=z^m h(e^{iw},z^{-1})$ which we write as a polynomial $$d(z)=\sum_{k=0}^m d_k z^k,$$ with coefficients $d_k=\sum_{j=0}^n h_{j,m-k} s^k,$ and $s=e^{iw}.$ \newline
\indent With the polynomial $d(z)$ we associate the Schur-Cohn $m \times m$ matrix $M = (d_{jk})$ specified by
$$d_{jk}=\sum_{l=1}^j(d_{m-j+l} \overline{d}_{m-k+l} - \overline{d}_{j-l} d_{k-l}), $$
where $j \leq k$. The matrix $M(e^{iw})$ is a Hermitian matrix and we define
$$g(e^{iw})=\det M(e^{iw}),$$ where $g(\cdot)$ is a self-inversive polynomial. \newline
\indent We state the following (see \cite{Siljak})
\begin{thm}\label{warunki}
A two-variable polynomial $h(s, z)$ has the stability property \ref{sta} if and only if
\begin{description}
  \item[(i)] $f(s)$ is $\mathbb{D}$-stable.
  \item[(ii)] $g(z)$ is $\mathbb{T}$-positive.
  \item[(iii)] $M(1)$ is positive definite,
\end{description} where $\mathbb{T}:=\{z \in \mathbb{C} \colon |z|=1\}$.
\end{thm}

It is easy observation, that $D_{q,r}^2$ is Lu Qi-Keng domain if and only if polynomial $F(\varepsilon x,  \varepsilon y)$ satisfies the stability property \ref{sta} for every $\varepsilon \in (0,1)$. \newline Let $\varepsilon \in (0,1)$, then
\begin{align*}
    F(\varepsilon x,  \varepsilon y)=& (\varepsilon^2 q r-2 \varepsilon^2 q-2 \varepsilon^2 r)xy + q r+2 q+2 r  \\
   & + (-\varepsilon q r-2 \varepsilon q+2 \varepsilon r)x + (-\varepsilon q r+2 \varepsilon q-2 \varepsilon r)y
\end{align*}
Now We will consider conditions (i), (ii) and (iii) from Theorem \ref{warunki} for  polynomial $F(\varepsilon x,  \varepsilon y)$, where $x, y \in \mathbb{C}$ are complex variables. \newline
\indent Condition (i) means, that the polynomial $$f(s)=(q r+2 q+2 r)s +2 \varepsilon r -\varepsilon q r-2 \varepsilon q $$ has all zeros in the open unit circle $\mathbb{D}$, which is equivalent, to state that the following inequalities
$$-1< \frac{\varepsilon q r+2 \varepsilon q- 2 \varepsilon r}{q r+2 q+2 r} < 1,$$ holds for every $q>0, r>0$. Simple calculations show that, these inequalities holds for any positive numbers $p$ and $r$. \newline \indent  In our case condition (iii) is included in condtion (ii), so we need only consider condition (ii). To test condition (ii), we consider polynomial
\begin{align*}
   g(z)=A^2 + B^2 -C^2 -D^2 + (A\cdot B - C \cdot D) (z+\overline{z}),
\end{align*} where $A=q r+2 q+2 r$, $B=-\varepsilon q r-2 \varepsilon q+2 \varepsilon r$, $C=-\varepsilon q r+2 \varepsilon q-2 \varepsilon r$, $D=\varepsilon^2 q r-2 \varepsilon^2 q-2 \varepsilon^2 r$. \newline
\indent Positivity of $g(z)$ on $\mathbb{T}$ (which is required by condition (ii) ) means, that following inequality
\begin{align}\label{n}
  (A\cdot B - C \cdot D)x>  C^2 + D^2- A^2 - B^2
\end{align} holds for every $q>0, r>0$ and $-1\leq x \leq1$. Easy calculation shows that inequality \ref{n} for every $q>0, r>0$ is true in cases when $x=-2$ or $x=2$. Which implies, that \ref{n} is true for every $q>0, r>0$ and $-1\leq x \leq1$. This completes the proof of the main theorem.
\section{Additional results}
Now We will consider following domains
$$D_{q,r}^1:= \{z \in \mathbb{C}^3 \colon |z_1| + |z_2|^q < 1, \quad |z_1| + |z_3|^r < 1\}.$$
Similarly as in section 3, we have

\begin{prop}\label{pr2} Let $\alpha_i \in \mathbb{Z}_{+}$ for $i=1,2,3$. Then, we have
$${\left\| z_1^{\alpha_1} z_2^{\alpha_2} z_3^{\alpha_3} \right\|}^2_{L^2(D_{q,r}^1)}= \frac{  \pi^3 \Gamma(2 \alpha_1 + 2) \Gamma(\frac{2\alpha_2 + 2}{q} +\frac{2\alpha_3 + 2}{r} + 1)  }{  (\alpha_2 + 1) (\alpha_3 +1 ) \Gamma( \frac{2\alpha_2 + 2}{q} +\frac{2\alpha_3 + 2}{r} + 2\alpha_1 + 3)} .$$
\end{prop}
By series representation of the Bergman kernel function, we have
\begin{align*}
K_{D_{q,r}^1} &((0, z_2, z_3),(0, w_2, w_3)) =\\& \frac{1}{2\pi^3} \sum_{\alpha_2 , \alpha_3 =0 }^{ \infty} \frac{(\alpha_2 + 1) (\alpha_3 +1 ) \Gamma( \frac{2\alpha_2 + 2}{q} +\frac{2\alpha_3 + 2}{r} + 3) }{\Gamma(2) \Gamma(\frac{2\alpha_2 + 2}{q} +\frac{2\alpha_3 + 2}{r} + 1)}  \nu_2^{\alpha_2} \nu_3^{\alpha_3} ,
\end{align*}
where  $\nu_2= z_2 \overline{w}_2$, $\nu_3= z_3 \overline{w}_3$.

Using the identity $\Gamma(a+1)=a \Gamma(a)$, after a little calculation, we obtain
\begin{align*}
K_{D_{q,r}^1} &((0, z_2, z_3),(0, w_2, w_3)) =  \frac{2 r^2 (\nu_2 (\nu_2+4)+1) (1-\nu_3)^2}{\pi^3 q^2 r^2 (1-\nu_2)^4 (1-\nu_3)^4} \\+& \frac{q^2 (1-\nu_2)^2 \left(-2 r^2 \nu_3 + (r-2) (r-1) \nu_3^2+(r+3) r+8 \nu_3+2\right)}{\pi^3 q^2 r^2 (1-\nu_2)^4 (1-\nu_3)^4} \\-& \frac{q r \left(1-\nu_2^2\right) (1-\nu_3) (3 r (\nu_3-1)-4 (\nu_3+1))}{\pi^3 q^2 r^2 (1-\nu_2)^4 (1-\nu_3)^4} ,
\end{align*}
where  $\nu_2= z_2 \overline{w}_2$, $\nu_3= z_3 \overline{w}_3$. \newline
Denote by
\begin{align*}
   G(x,y)=&q^2 (x-1)^2 \left(-2 r^2 y+(r-2) (r-1) y^2+(r+3) r+8 y+2\right)\\
 &- q r \left(x^2-1\right) (y-1) (3 r (y-1)-4 (y+1))\\&+2 r^2 (x (x+4)+1) (y-1)^2,
\end{align*}
then the Bergman kernel $K_{D_{q,r}^1}$ has zero inside $D_{q,r}^1 \times D_{q,r}^1$ if polynomial $G(\varepsilon x, \varepsilon y)$ does not satisfy the stability property \ref{sta} for some $0<\varepsilon<1$. \newline
\indent Now We will consider conditions (i) from Theorem \ref{warunki} for  polynomial $G(\varepsilon x,  \varepsilon y)$ in case when $q=r$.  If $q=r$, then we have
\begin{align*}
   G(\varepsilon x, \varepsilon y)=& A_{22} x^2 y^2 + A_{21} x^2y +A_{20} x^2  + A_{12} xy^2 \\&  +  A_{11}xy + A_{10} x + A_{02} y^2+ A_{01} y + A_{00} ,
\end{align*}
where $A_{22}=\varepsilon^4 (8 - 6 r + r^2)$, $A_{21}=\varepsilon^3 (4 + 6 r - 2 r^2)$, $A_{20}=\varepsilon^2 r^2$, $A_{12}=A_{21}$, $A_{11}=\varepsilon^2 4 (r^2-32)$, $A_{10}=\varepsilon (-2 r^2-6 r+4)$, $A_{20}=A_{02}$, $A_{10}=A_{01}$, and  $A_{00}=r^2+6 r+8$. \newline
\indent  Condition (i) means, that the polynomial $$f(s)=(r^2+6 r+8)s^2 + \varepsilon (-2 r^2-6 r+4)s + \varepsilon^2 r^2$$ has all zeros in the open unit circle $\mathbb{D}$, which is equivalent, to state that the  transformed polynomial $$Q(s)=(s-1)^2 f\left(\frac{s+1}{s-1}\right)$$ is Hurwitz stable  and $f(1)\neq 0$. It is easy observation, that $$f(1)=(1-\varepsilon)^2 r^2 + 6(1-\varepsilon)r +8 +4 \varepsilon >0$$ if $0<\varepsilon<1$ and $r>0$. Polynomial $Q(s)$ is  Hurwitz stable if and only if (see \cite{ZY} for details) $$\frac{2(1-\varepsilon^2)r^2+12r+16}{(1-\varepsilon)^2 r^2 + 6(1-\varepsilon)r +8 +4 \varepsilon} >0$$
and $$\frac{(1+\varepsilon)^2 r^2 + 6(1+\varepsilon)r +8 -4 \varepsilon}{(1-\varepsilon)^2 r^2 + 6(1-\varepsilon)r +8 +4 \varepsilon}>0$$ which is true when $0<\varepsilon<1$ and $r>0$. Hence $f(s)$ is $\mathbb{D}$-stable. \newline
To test condition (ii), we consider polynomial
\begin{align*}
   d(z)=&(A_{00} + A_{10}t + A_{20}t^2 )z^2 + (A_{01} + A_{11}t + A_{21}t^2) z \\&+ A_{02} + A_{12}t + A_{22}t^2,
\end{align*} where $t=e^{iw}$.
With the polynomial $d(z)$ we associate the Schur-Cohn $2 \times 2$ matrix
$$M(t)=\left[
  \begin{array}{cc}
    d_2 \overline{d}_2 - d_0 \overline{d}_0  &  d_2 \overline{d}_1 -  \overline{d}_0 d_1  \\
    \overline{d}_2 d_1 -  d_0 \overline{d}_1 & d_2 \overline{d}_2 - d_0 \overline{d}_0  \\
  \end{array}
\right],$$ \indent where $d_2=A_{00} + A_{10}t + A_{20}t^2$ , $d_1=A_{01} + A_{11}t + A_{21}t^2$ and $d_0=A_{02} + A_{12}t + A_{22}t^2$, $t=e^{iw}$. Now we define $g(t)=\det M(t)$. After some calculation for $\varepsilon=1$ (with the help of a computer program Maple or Mathematica), we have
$$g(e^{iw})=27648 r^{10} (\eta-1)^3 \left(r^2 (\eta-1)+4\right),$$  where $\eta=\cos w$. It is easy to see that for every $r>0$ there exist $\eta < 1$ such that $g(e^{iw})<0$. Hence there exist $1>\varepsilon >0$, such that polynomial $G(\varepsilon x, \varepsilon y)$ does not satisfy the stability property \ref{sta}. Therefore the Bergman kernel function for $K_{D_{r,r}^1}$ is not zero free. As a consequence of above consideration, we have following proposition
\begin{prop}\label{pro3}
For any $r>0$, domain $D_{r,r}^1$ defined by
$$D_{r,r}^1:= \{z=(z_1,z_2,z_3) \in \mathbb{C}^3 \colon |z_1| + |z_2|^r < 1, \quad |z_1| + |z_3|^r < 1\}$$ is not Lu Qi-Keng.
\end{prop}
\section{Some remarks}
 As stated in Introduction, Bergman kernel for $\Omega^{(p,1/2)}_{1,1}= \{(w, z) \in  \mathbb{C} \times \mathbb{C} \colon |w|^{2p} + |z| < 1\}$ has zeroes if and only if $1/p > 2$. Moreover it is well known fact, that the Thullen domain $\Omega^{(p,1)}_{1,1} = \{(w, z) \in  \mathbb{C} \times \mathbb{C} \colon |w|^{2p} + |z|^2 < 1\}$ is a Lu Qi-Keng domain for $p > 0$. \newline
\indent In view of the Proposition \ref{pro3} and Theorem \ref{MT}, we can ask the following question: Is there a relationship between the existence of zeros of  the Bergman kernel function for  domains
\begin{align*}
   \{z \in \mathbb{C}^3 \colon |z_1|^p + |z_2|^q < 1, \quad |z_1|^p + |z_3|^q < 1 \}\\
 \text{and} \quad \{z \in \mathbb{C}^2 \colon |z_1|^p + |z_2|^q < 1\} \,?
\end{align*}
 Through Proposition \ref{pro3} we know, that existence of zeros of  the Bergman kernel function for $\{z \in \mathbb{C}^3 \colon |z_1|^p + |z_2|^q < 1, \quad |z_1|^p + |z_3|^q < 1 \}$ does not imply existence of zeros of  the Bergman kernel function for $\{z \in \mathbb{C}^2 \colon |z_1|^p + |z_2|^q < 1\}$ in general. It is interesting question whether the converse is true?

\small{

\bibliographystyle{amsplain}
}
\noindent Tomasz Beberok\\
Department of Applied Mathematics\\
University of Agriculture in Krakow\\
ul. Balicka 253c, 30-198 Krakow, Poland\\
email: tbeberok@ar.krakow.pl
\end{document}